\documentclass[11pt,a4paper]{scrartcl}
\usepackage[english]{babel}
\usepackage{a4wide}
\usepackage[T1]{fontenc}
\usepackage[utf8]{inputenc}
\usepackage{amsmath, amsthm, amssymb}
\usepackage[numbers]{natbib}
\usepackage{graphicx}
\usepackage{url}
\usepackage{hyperref}
\usepackage{xcolor}
\usepackage{booktabs}

\newcommand{\norm}[1]{\lVert #1 \rVert}

\newtheorem{theorem}{Theorem}
\newtheorem{lemma}{Lemma}
\newtheorem{remark}{Remark}

\title{On Rational Krylov and Reduced Basis Methods for Fractional Diffusion}

\author{Tobias Danczul\thanks{Institute for Analysis and Scientific Computing (TU Wien),
		Wiedner Hauptstrasse 8-10, 1040 Wien, Austria.
		\texttt{<tobias.danczul@tuwien.ac.at>}}
    \and
    Clemens Hofreither\thanks{
    Johann Radon Insitute for Computational and Applied Mathematics (RICAM),
    Altenbergerstr.~69, 4040 Linz, Austria.
    \texttt{<clemens.hofreither@ricam.oeaw.ac.at>}
}}

\begin{document}

\maketitle

\begin{abstract}
    We establish an equivalence between two classes of methods for solving
    fractional diffusion problems, namely, Reduced Basis Methods (RBM) and
    Rational Krylov Methods (RKM). In particular, we demonstrate that
    several recently proposed RBMs for fractional diffusion can be interpreted as RKMs. This
    changed point of view allows us to give convergence proofs for some
    methods where none were previously available.
    
    We also propose a new RKM for fractional diffusion problems with poles
    chosen using the best rational approximation of the function
    $x^{-s}$ in the spectral interval of the spatial discretization
    matrix. We prove convergence rates for this method and demonstrate
    numerically that it is competitive with or superior to many methods
    from the reduced basis, rational Krylov, and direct rational
    approximation classes.
    We provide numerical tests for some elliptic fractional diffusion
    model problems.
\end{abstract}

\section{Introduction}

The area of numerical methods for diffusion problems with a fractional in
space diffusion operator has seen intensive development recently. In the
present work, our interest lies in the spectral definition of the
fractional diffusion operator in bounded domains with homogeneous Dirichlet
boundary conditions.
Numerical treatment of such problems by extension to a higher-dimensional, but
local diffusion problem was proposed and analyzed in \cite{Nochetto2015}.
Several different methods \cite{Bonito2015,Aceto2019b} make use of quadrature
formulae for Dunford-Taylor or related integral representations of the
fractional power of the diffusion operator.
A reformulation of the fractional problem as a pseudo-parabolic equation and
solving it via a time-stepping scheme has been proposed in
\cite{Vabishchevich2015,Vabishchevich2016}.
Methods based on best uniform rational approximation (BURA) of certain
functions in the spectral domain were developed in
\cite{BURA:2019,Harizanov2020}.

It is remarkable that all the above-mentioned methods can viewed as rational
approximation methods, where the exact fractional power of the involved
operator is approximated by a rational function of the diffusion operator. This
unified view proposed in \cite{H:2020} has led to several interesting
ramifications for analysis and efficient implementation of these approaches.

A different class of numerical methods results from applying so-called
rational Krylov methods (cf.~\cite{Guettel:PhD,Guettel2013,Druskin2009b}) to
the solution of fractional diffusion problems \cite{Moret2018,Aceto2019}.
This approach can also be viewed as a rational approximation of the fractional
operator, but whereas the denominator of the rational function is fixed (via
selection of the poles) \textit{a priori}, the numerator is determined
automatically via Rayleigh-Ritz extraction, yielding a quasi-optimal
approximation from the rational Krylov space.

Several recently proposed numerical schemes exploit the fact that the non-local character of the
fractional operator can be circumvented at the cost of parametric solutions to
classical reaction-diffusion problems. The reduced basis method (RBM; see
\cite{Quateroni2016,Rozza2015}) is a prevalent choice for reducing the
computational effort in the evaluation of these solutions for multiple instances
of the parameter. Due to its usability and excellent convergence properties, the RBM has been applied to the extension method \cite{Antil2018}, the framework of interpolation operators \cite{DS:2019,DS:2020}, and quadrature approximations based on Dunford-Taylor calculus \cite{Bonito2020, Antil2019}; see also \cite{Gunzburger2020}. The analytical results provided by \cite{Bonito2020, DS:2019, DS:2020} underpin the experimental observations in \cite{Gunzburger2016} that the RBM has the ability to efficiently query the solution map for multiple values of the fractional exponent.

The aim of the present work is to establish a close relationship between
rational Krylov methods and reduced basis methods for fractional
diffusion problems. We will show that several reduced basis methods can be
interpreted as rational Krylov methods. This in turn allows us to apply a
strong result on quasi-optimality of rational Krylov methods to the convergence
analysis of reduced basis method, yielding novel error estimates.
In a sense, this continues the work started in \cite{H:2020}, where a unified
theoretical framework for direct rational approximation methods was proposed,
in that we now extend this unifying point of view also to rational Krylov
and reduced basis methods.

The remainder of the paper is laid out as follows:
In Section~\ref{sec:prelim}, we recall the spectral version of the
fractional diffusion problem and its discretization.
In Section~\ref{sec:methods}, we describe several classes of numerical
methods for the efficient solution of fractional diffusion problems. In
particular, we draw some parallels between rational Krylov and reduced basis
methods. Furthermore, we propose a rational Krylov method based
on the poles of the best rational approximation and analyze its convergence.
We derive some new theoretical convergence results for a reduced basis method
in Section~\ref{sec:analytical} by making use of the rational Krylov framework.
Finally some numerical experiments are given in Section~\ref{sec:numerics}.

\section{The fractional diffusion problem and its discretization}
\label{sec:prelim}

Given an open and bounded domain $\Omega \subset \mathbb R^d$,
$s \in (0,1)$, and a suitable right-hand side $b$ defined on $\Omega$,
we seek the solution $u$ of the fractional diffusion equation
\begin{equation}
    \label{eq:fracdiff}
        \mathcal L^s u = b \qquad \text{in } \Omega
\end{equation}
where $\mathcal L u = -\operatorname{div}(A \nabla u)$ is a self-adjoint,
elliptic diffusion operator with $A(x) \in \mathbb R^{d \times d}$
symmetric and uniformly positive definite, and $\mathcal L$ is supplemented
with homogeneous Dirichlet boundary conditions on
$\Gamma = \partial\Omega$.

Different definitions of fractional powers of operators in bounded
domains exist (see, e.g., \cite{Lischke2020} and its references).
In the present work, we assume the following spectral definition.
Under mild assumptions on the operator and the boundary, $\mathcal L$ admits a system of
eigenfunctions $u_j$ with corresponding eigenvalues $\widetilde\lambda_j>0$ such that
\[
    \mathcal L u_j = \widetilde\lambda_j u_j \qquad \forall j=1,2,\ldots
\]
and $(u_i, u_j) = \delta_{ij}$, where $(\cdot,\cdot)$ denotes the $L_2$-inner
product in $\Omega$. A fractional power of $\mathcal L$ can then be defined
as
\begin{equation}
    \label{eq:spectral}
    \mathcal L^s u = \sum_{j=1}^\infty \widetilde\lambda_j^s (u, u_j) u_j.
\end{equation}

In order to discretize this problem, we introduce a finite-dimensional space
$V_h \subset H^1_0(\Omega)$, for instance constructed using finite elements,
together with a suitable basis $(\varphi_j)_{j=1}^n$.
We introduce the standard stiffness and mass matrices $K$ and
$M$, respectively, as
\begin{align}
	\label{eq:matrices}
    K_{ij} = (A \nabla \varphi_j, \nabla \varphi_i),    \quad
    M_{ij} = (\varphi_j, \varphi_i)
    \qquad
    \forall i,j=1,\ldots,n,
\end{align}
and let $L = M^{-1} K \in \mathbb R^{n \times n}$.
A simple but computationally expensive way to solve problems of
the form \eqref{eq:fracdiff} is via the fractional matrix power
\begin{equation}
    \label{eq:dem}
    \mathbf u = L^{-s} \mathbf b,
\end{equation}
where $\mathbf b \in \mathbb R^n$ is the coefficient vector of the
$L_2$-projection of the right-hand side $b$ into $V_h$.
The resulting vector $\mathbf u \in \mathbb R^n$ contains the coefficients of
the discrete solution with respect to the basis $(\varphi_j)$.
In \cite{H:2020}, it was shown that this discrete formulation is
equivalent to the discrete eigenfunction method which replaces the
exact eigenfunctions and eigenvalues in \eqref{eq:spectral} with
their discrete counterparts obtained by solving the generalized
eigenvalue problem
\begin{equation}
    \label{eq:discr_eig}
    K \mathbf u_j = \lambda_j M \mathbf u_j,
    \qquad j=1,\dots,n,
\end{equation}
where we assume that $\lambda_1 \le \lambda_2 \le\ldots\le \lambda_n$.
On the other hand, in \cite{DS:2019,DS:2020} it was demonstrated that the
solution obtained by \eqref{eq:dem} is also equivalent to a number of
different interpolation constructions between the finite-dimensional
Hilbert spaces
\[
    (V_h, \norm{\cdot}_{L_2}), \ (V_h, \norm{\cdot}_{A})
\]
with argument $s \in (0,1)$, where
$\norm{u}_{A} = (A \nabla u, \nabla v)^{1/2}$ is the energy norm.

Realizing \eqref{eq:dem} exactly is too computationally
expensive if the involved matrices $K$ and $M$ are large and sparse as it
involves the computation of the entire eigensystem \eqref{eq:discr_eig}.
Therefore, a number of approximation techniques have been developed,
a few of which we will outline in the following section.

\section{Approximation methods for fractional diffusion}
\label{sec:methods}

\subsection{Rational approximation methods}
\label{sec:ratapprox}

One class of methods presupposes that we have a rational function $r$ of
degree at most $k$ which in some sense approximates the function
$z \mapsto z^{-s}$ on the spectral interval
$\Lambda = [\lambda_{\min}(L), \lambda_{\max}(L)] = [\lambda_1,\lambda_n]$.
The idea is to approximate $L^{-s}$ by $r(L)$ in \eqref{eq:dem}.
To facilitate this, assume further that $r$ has the partial fraction
decomposition
\[
    r(z) = c_0 + \sum_{j=1}^k \frac{c_j}{z - d_j}
\]
with real, nonpositive, and pairwise distinct poles $(d_j)_{j=1}^k$ and
residues $(c_j)_{j=0}^k$.
Then the application of the matrix function $r(L)$ to $\mathbf b$
is given by
\[
    \mathbf u_r := r(L) \mathbf b =
    c_0 \mathbf w_0 + \sum_{j=1}^k c_j \mathbf w_j,
    \qquad
    \mathbf w_0 = \mathbf b,
    \qquad
    \mathbf w_j = (L - d_j I)^{-1} \mathbf b, \qquad j=1,\dots,k,
\]
or equivalently
\begin{equation}
    \label{eq:shifted}
    (K - d_j M) \mathbf w_j = M \mathbf b, \qquad j=1,\dots,k.
\end{equation}
The error of the solution so obtained relative to the solution
$\mathbf u$ from \eqref{eq:dem} can be bounded directly in terms
of the approximation quality of $r$ to the function $z \mapsto z^{-s}$,
as the following result shows.

\begin{theorem}[\cite{H:2020}]
    \label{thm:rational_error}
    The solution $u_r \in V_h$ obtained by the rational approximation method
    and the solution $u \in V_h$ obtained by the discrete eigenfunction method
    satisfy the relation
    \[
        \norm{u - u_r}_{L_2(\Omega)}
        \le
        \norm{b}_{L_2(\Omega)}
        \norm{f - r}_{L_\infty(\Lambda)},
    \]
    where $f(z)=z^{-s}$.
\end{theorem}

Conversely, any vector
$\mathbf u \in \operatorname{span}\{\mathbf b, \mathbf w_1, \dots, \mathbf w_k\}$,
where the $\mathbf w_k$ are obtained as the solutions \eqref{eq:shifted}
of shifted diffusion problems with pairwise distinct shifts $d_j$,
can be written as $r(L) \mathbf b$ with a rational function $r$
of degree at most $k$ with poles $(d_j)$. Thus many numerical approaches
for solving fractional diffusion problems which involve the solution
of such shifted problems can be recast as rational approximation
methods, as has been systematically studied in \cite{H:2020}.

\subsection{Rational Krylov methods}
\label{sec:ratkry}

A variant of direct rational approximation is given by the so-called
rational Krylov methods. Here the idea is to specify the poles
$(d_j)$ of the involved rational approximation \textit{a priori}, but
determine the coefficients $(c_j)$ by solving a reduced problem
in the so-called rational Krylov space.
Thus, we fix the poles
$(d_j)_{j=0}^k\subset\overline{\mathbb{R}}_0^- := \mathbb{R}_0^-\cup\{-\infty\}$
and introduce the associated polynomial
\begin{equation}
\label{eq:q}
    q_{k}(z) := \prod_{\substack{j=0 \\ d_j\neq -\infty}}^k (z - d_j) \in \mathcal P_{k+1},
\end{equation}
where $\mathcal P_{k+1}$ denotes the algebraic polynomials of degree at most
$k+1$. Following \cite{Guettel:PhD,Guettel2013}, we define the \emph{rational
Krylov space}
\[
    \mathcal{Q}_{k+1} := \mathcal Q_{k+1}(L, \mathbf b) :=
    q_k(L)^{-1} \mathcal K_{k+1}(L, \mathbf b) \subset \mathbb R^n,
\]
where
\[
    \mathcal K_{k+1}(L, \mathbf b)
    = \operatorname{span} \{ \mathbf b, L \mathbf b, \ldots, L^k \mathbf b \}
    = \mathcal P_k(L) \mathbf b
\]
is the standard (polynomial) Krylov space and we denote by $\mathcal P_k(L)$
the space of polynomials in $L$ of degree at most $k$.
Some fundamental properties of rational Krylov spaces are given
in the following lemma.
\begin{lemma}
    The rational Krylov space $\mathcal Q_{k+1}(L, \mathbf b)$
    has the properties
    \begin{enumerate}
        \item
            $\mathcal Q_{k+1}(L,\mathbf{b}) = \mathcal K_{k+1}(L, q_k(L)^{-1} \mathbf b)
            = \mathcal P_k(L) q_k(L)^{-1} \mathbf b$,
        \item
            if $d_j = -\infty$ for some $j\in\{0,\dots,k\}$, then $\mathbf b \in \mathcal Q_{k+1}(L,\mathbf{b})$,
        \item
            $\dim(\mathcal Q_{k+1}(L, \mathbf b)) =
             \dim(\mathcal K_{k+1}(L, \mathbf b)) =
             \min\{k+1, M\}$,
            where $M$ is the \emph{invariance index} (see \cite{Guettel:PhD}) of the Krylov space
            $\mathcal K_{k+1}(L, \mathbf b)$.
    \end{enumerate}
    \label{lem:ratkry}
\end{lemma}
\begin{proof}
See \cite[Lemma 4.2]{Guettel:PhD} for the case $d_j = -\infty$ for one $j\in\{0,\dots,k\}$. If all poles are finite, the claims are validated analogously.
\end{proof}


The connection to the rational approximation methods sketched in
Section~\ref{sec:ratapprox} is easily established.
If we let
\[
    p_i(z) := \prod_{\substack{j\neq i \\ d_j\neq -\infty}}(z - d_j) \in \mathcal P_{k},
    \qquad i=0,\dots,k,
\]
and agree on the convention
$(L-d_j I)^{-1}\mathbf{b}:= \mathbf{b}$ for $d_j = -\infty$,
we see that the vectors introduced in \eqref{eq:shifted} satisfy
\[
    \mathbf w_j = (L - d_j I)^{-1} \mathbf b = p_j(L) q_k(L)^{-1} \mathbf b
\]
and therefore $\mathbf w_j \in \mathcal{Q}_{k+1}$ due to the first property.
Thus, if the vectors $(\mathbf w_0, \dots, \mathbf w_k)$ are linearly
independent, it follows from the third property that
$\dim(\mathcal{Q}_{k+1}) = k+1 \le M$ and
\begin{align}
	\label{eq:spanQ}
    \mathcal{Q}_{k+1} = \operatorname{span} \{ \mathbf w_0, \dots, \mathbf w_k \}.
\end{align}
In other words, if the chosen poles are pairwise distinct,
the rational Krylov space is identical to the space spanned by the
solutions of the shifted problems \eqref{eq:shifted}.

An orthonormal basis for the rational Krylov space $\mathcal{Q}_{k+1}$ is typically
computed using the \emph{rational Arnoldi method} \cite{Ruhe1984,Guettel:PhD}.
This algorithm requires as its input $L$, the right-hand side $\mathbf b$ and
the poles $(d_j)_{j=1}^k$.  It entails solving shifted problems similar to
\eqref{eq:shifted} and then orthonormalizes the resulting vectors, resulting in
a matrix $W \in \mathbb R^{n \times(k+1)}$ with orthonormal columns which spans
$\mathcal{Q}_{k+1}$. For a given scalar function $f$ defined over $\Lambda$,
an approximation $\mathbf u_{k+1}$ to
the vector $f(L) \mathbf b$ within this subspace is then found via
\emph{Rayleigh-Ritz extraction}, namely
\begin{align}
	\label{eq:rkm}
	\mathbf u_{k+1} := W f(L_{k+1}) W^T \mathbf b \in\mathcal{Q}_{k+1}(L,\mathbf{b}),
	\qquad
	L_{k+1} := W^T L W \in \mathbb R^{(k+1)\times(k+1)}.
\end{align}
The matrix $L_{k+1}$ is typically much smaller than $L$, and thus $f(L_{k+1})$
can be computed, e.g., by diagonalization.
Güttel \cite{Guettel:PhD} proves that this procedure is basis-independent, that
is, $\mathbf u_{k+1}$ depends only on the space $\mathcal Q_{k+1}(L,\mathbf b)$,
not the matrix $W$ itself.
Furthermore he points out that Rayleigh-Ritz extraction is equivalent to Galerkin projection in the special case $f(z) = z^{-1}$, i.e., when solving a linear system with the matrix $L$.

The rational Krylov space $\mathcal{Q}_{k+1}$ and its basis representation $W$
depend on $L$ and $\mathbf b$, rendering the above procedure nonlinear,
but the same is true for standard Krylov space methods. In contrast, the
direct rational approximation methods described in Section~\ref{sec:ratapprox}
are linear since they are given by $\mathbf u_r = r(L) \mathbf b$ with $r$
fixed \textit{a priori}. Nevertheless, we can also find a rational representation
of this form for the rational Krylov method if we allow $r$ to depend on
the input data, as the following result shows.

\begin{theorem}
    \label{thm:krylov_interpol}
    The solution obtained by the rational Krylov method satisfies
    \[
        \mathbf u_{k+1} = r(L) \mathbf b,
    \]
    where $r = p / q_{k}$ and
    $p \in \mathcal P_k$ is a polynomial such that $r$ satisfies
    the interpolation conditions
    \[
        r(\mu_j) = f(\mu_j), \qquad j=1,\dots,k+1,
    \]
    where the \emph{rational Ritz values} $(\mu_j)_{j=1}^{k+1}$ are the
    eigenvalues of $L_{k+1}$.
\end{theorem}
\begin{proof}
	See \cite[Theorem 4.8]{Guettel:PhD} for the case $d_j = -\infty$ for one $j = 0,\dots,k$.
    If all poles are finite, the proof follows analogously.
\end{proof}
Since the denominator $q_k$ of $r$ is fixed, $p$ is determined by the
polynomial interpolation problem
$p(\mu_j) = q_k(\mu_j) f(\mu_j)$ for $j=1,\dots,k+1$. If the rational Ritz
values $\mu_j$ are pairwise distinct, $p$ is uniquely determined by these
conditions.

Clearly, the quality of the approximation $\mathbf u_{k+1} \in \mathcal{Q}_{k+1}$ to
$f(L) \mathbf b$ depends on the rational Krylov space and therefore on a
suitable choice of the poles $(d_j)_{j=0}^k$. However, the following powerful
result shows that within this space, the approximation is quasi-optimal.
Here we write $\mathbb W(L)$ for the numerical range of $L$ which, in particular,
contains the spectrum of $L$,
and $\norm{\cdot}$ refers to the Euclidean vector norm.

\begin{theorem}
    Let $W$ be an orthonormal basis of $\mathcal{Q}_{k+1}(L,\mathbf{b})$ and $L_{k+1} := W^T L W$.
    Let $f$ be analytic in a neighborhood of $\mathbb W(L)$ and
    $\mathbf u_{k+1} = W f(L_{k+1}) W^T \mathbf b$.
    For every set $\Sigma \supseteq \mathbb W(L)$ there holds
    \[
        \norm{f(L) \mathbf b - \mathbf u_{k+1}} \le
        2 C \norm{\mathbf b} \min_{p \in \mathcal P_k} \norm{f - p/{q_k}}_{L_\infty(\Sigma)}
    \]
    with a constant $C \le 11.08$.
    If $L$ is self-adjoint, the result holds with $C=1$.
    \label{thm:rkm_opt}
\end{theorem}
\begin{proof}
	See \cite[Theorem 4.10]{Guettel:PhD} and \cite[Proposition 3.2]{Druskin2009}.
\end{proof}

Note the close relation of this result to Theorem \ref{thm:rational_error}:
roughly, the error obtained using the rational Krylov method with given poles
$(d_j)$ is not much larger than the error obtained using the best possible
rational approximation method with a rational function $r$ having these same
poles.

\subsection{Reduced basis methods}
\label{sec:rmb}

In this section, we show that several recently proposed schemes which are based on RBMs admit a representation in the rational Krylov framework.
To make matters precise, we consider the discrete parametric reaction-diffusion equation
\begin{align}
	\label{eq:reacdiff}
	(tI+L)\mathbf{w}(t) = \mathbf{b}
\end{align}
for a prescribed right-hand side $\mathbf{b}\in\mathbb{R}^n$ and a parameter $t\in\overline{\mathbb{R}}_0^+ := \mathbb{R}_0^+ \cup\{\infty\}$ that encodes the variability of the problem. We set $\mathbf{w}(\infty) := \mathbf{b}$ by convention. The RBM seeks to approximate the manifold of solutions $(\mathbf{w}(t))_{t\in\mathbb{R}_0^+}$ in the low-dimensional space
\begin{align}
	\label{eq:reducedspace}
	\mathcal{V}_{k+1} := \mathcal{V}_{k+1}(L,\mathbf{b}) := \operatorname{span}\{\mathbf{w}(t_0),\dots,\mathbf{w}(t_k)\},
\end{align}
where $0\leq t_0 < \dots < t_k$ are particular parameters which we refer to as \textit{snapshots}\footnote{Our terminology differs from standard RBM notation, where the term \textit{snapshot} is typically employed to refer to the discrete solution $\mathbf{w}(t_j)$ instead of the parameter $t_j$ itself.} throughout this manuscript. The reduced basis analogon of the last claim in Lemma \ref{lem:ratkry} is provided in \cite[Lemma 3.5]{DS:2019}: it states that $\operatorname{dim}(\mathcal{V}_{k+1}) = k+1$ if $\mathbf{b}$ is excited by sufficiently many eigenfunctions of $L$. The reduced basis surrogate $\mathbf w_{k+1}(t) \in \mathcal V_{k+1}$ for $\mathbf{w}(t)$ is computed via Galerkin projection,
\begin{align}
	\label{eq:rbm}
    \mathbf{w}_{k+1}(t) := V(tI_{k+1} + L_{k+1})^{-1}V^T\mathbf{b} \in \mathcal{V}_{k+1}(L,\mathbf{b}),
    \qquad
    L_{k+1} := V^TLV\in \mathbb R^{(k+1)\times(k+1)},
\end{align}
where $I_{k+1}\in\mathbb{R}^{(k+1)\times(k+1)}$ denotes the identity matrix and $V\in\mathbb{R}^{n\times(k+1)}$ a matrix whose columns form an orthonormal basis of $\mathcal{V}_{k+1}(L,\mathbf{b})$. After an initial computational investment, the reduced space \eqref{eq:reducedspace} allows us to evaluate the coefficient vector of $\mathbf{w}_{k+1}(t)$ in the basis $\{\mathbf{w}(t_0),\dots,\mathbf{w}(t_k)\}$ for arbitrary $t$ with complexity only depending on $k$. Due to \eqref{eq:spanQ}, we immediately obtain the following result.
\begin{lemma}
	\label{lemma:equal_searchspace}
	Let $(t_j)_{j=0}^k\subset\overline{\mathbb{R}}_0^+$ be pairwise distinct. Then the reduced space $\mathcal{V}_{k+1}(L,\mathbf{b})$ with snapshots $(t_j)_{j=0}^k$ and the rational Krylov space $\mathcal{Q}_{k+1}(L,\mathbf{b})$ with poles $(-t_j)_{j=0}^k$ coincide.
\end{lemma}

In the following two subsections, we study two classes of reduced basis
methods which have been applied to the fractional diffusion problem, namely
ones based on interpolation and on quadrature, and establish their connection
to rational Krylov methods.

\subsubsection{Interpolation-based reduced basis methods}

Two different model order reduction strategies have been recently proposed in \cite{DS:2020} which couple interpolation theory with reduced basis technology. In line with \cite{DS:2019}, the (forward) fractional operator with positive exponent $s\in(0,1)$ is reinterpreted as a weighted integral over parametrized reaction-diffusion problems
\begin{align}
	\label{eq:int_repr}
	L^s \mathbf b = \frac{2\sin(\pi s)}{\pi}\int_0^\infty t^{-2s-1}(\mathbf{b} - \mathbf v(t))\, dt, \qquad\mathbf v(t) := (I+t^2L)^{-1}\mathbf{b}.
\end{align}
Invoking $\mathbf{v}(t) = t^{-2}\mathbf{w}(t^{-2})$ and the substitution $y = t^{-2}$, where we rename the substituted variable $t$ again, we observe that the integrand can be expressed in terms of the parameter family $\mathbf w(t)$ via
\begin{align*}
	L^{s}\mathbf{b} = \frac{\sin(\pi s)}{\pi}\int_0^{\infty}t^{s}(t^{-1}\mathbf{b}-\mathbf{w}(t))\, dt.
\end{align*}
Based on a selection of snapshots $(t_j)_{j=0}^k\subset\overline{\mathbb{R}}_0^+$, the integrand is approximated using a RBM, yielding
\begin{align*}
	\mathbf{u}^{\textsc{Int}}_{k+1} := \frac{\sin(\pi s)}{\pi}\int_0^{\infty}t^{s}(t^{-1}\mathbf{b} - \mathbf{w}_{k+1}(t))\, dt.
\end{align*}
As shown in \cite[Theorem 4.3]{DS:2019}, the surrogate evaluates to
\begin{align}
	\label{eq:rbm_computation}
	\mathbf{u}_{k+1}^{\textsc{Int}} = V L_{k+1}^sV^T\mathbf{b}
\end{align}
where $V$ refers to a matrix of orthonormal basis vectors of $\mathcal{V}_{k+1}$. In \cite{DS:2019} it was proven that the scheme approximates $L^s\mathbf{b}$ at exponential convergence rates. Motivated by these results, the authors of \cite{DS:2020} proposed a version of \eqref{eq:rbm_computation} for the backward operator. They confirmed experimentally that
\begin{align}
	\label{eq:generic_rbm}
	\mathbf{u}_{k+1}^{\textsc{RB}} := V f(L_{k+1})V^T\mathbf{b}
\end{align}
converges exponentially to $L^{-s}\mathbf{b}$ if $f(z) = z^{-s}$, but no rigorous proof was known so far. The following theorem provides the essential tool to close this gap in the literature and allows us to establish a connection to RKMs.
\begin{theorem}
	\label{thm:equivalence}
	Let $f$ be analytic in a neighbourhood of $\mathbb{W}(L)$ and $(t_j)_{j=0}^k\subset\overline{\mathbb{R}}_0^+$ pairwise distinct. Then the reduced basis approximation $\mathbf{u}^{\textsc{RB}}_{k+1}$ with snapshots $(t_j)_{j=0}^k\subset\overline{\mathbb{R}}_0^+$ coincides with the rational Krylov approximation \eqref{eq:rkm} with poles $(-t_j)_{j=0}^k$.
\end{theorem}
\begin{proof}
    As pointed out by G\"uttel in \cite[Lemma 3.3]{Guettel:PhD}, the rational Krylov approximation is independent of the choice of the particular basis. In view of \eqref{eq:rkm} and \eqref{eq:generic_rbm}, it thus suffices to verify that the corresponding search spaces $\mathcal Q_{k+1}(L,\mathbf{b})$ and $\mathcal V_{k+1}(L,\mathbf{b})$ coincide. This is true due to Lemma \ref{lemma:equal_searchspace}.
\end{proof}

The second method presented in \cite{DS:2020}, also referred to as dual reduced basis approximation, follows a similar idea but is based on $L^{-1}$. Due to Theorem 2.2, 2.3, and Lemma 2.7 in \cite{DS:2020}, the negative fractional operator can be expressed as
\[
	L^{-s}\mathbf{b} = \frac{2\sin(\pi s)}{\pi}\int_0^\infty t^{-2s-1}(L^{-1}\mathbf{b} - (L^{-1}+t^2I)^{-1}L^{-1}\mathbf{b}) \, dt.
\]
Utilizing $y = t^{-2}$ while renaming the substituted variable $t$ again, we obtain
\begin{align*}
	L^{-s}\mathbf{b} &= \frac{\sin(\pi s)}{\pi}\int_0^\infty t^s(t^{-1}L^{-1}\mathbf{b} - (tL^{-1} + I)^{-1}L^{-1}\mathbf{b})\, dt \\
	&= \frac{\sin(\pi s)}{\pi}\int_0^\infty t^s(t^{-1}L^{-1}\mathbf{b} - \mathbf{w}(t))\, dt.
\end{align*}
The latter is again approximated utilizing reduced basis technology with prescribed snapshots $(t_j)_{j=0}^k\subset\overline{\mathbb{R}}_0^+$ by means of
\begin{align}
	\label{eq:dualrbm}
	\mathbf{u}^{\textsc{Dual}}_{k+1} := \frac{\sin(\pi s)}{\pi}\int_0^\infty t^s(t^{-1}L^{-1}\mathbf{b} - \mathbf{w}_{k+1}(t))\, dt.
\end{align}
In \cite[Theorem 3.4]{DS:2020} it has been shown that \eqref{eq:dualrbm} can be computed via
\begin{align}
	\label{eq:dualcomputation}
	\mathbf{u}^{\textsc{Dual}}_{k+1}  = L^{-1}V L_{*,k+1}^{s-1}V^T\mathbf{b}, \qquad  L_{*,k+1} := V^TL^{-1}V.
\end{align}
If $t_j = \infty$ for one $j\in\{0,\dots,k\}$, the surrogate can be interpreted as a post-processed rational Krylov approximation as follows.
\begin{theorem}
	Let $s\in(0,1)$, $f(z) = z^{s-2}$, $\mathbf{u}_{k+1}^{\textsc{Dual}}$ the dual reduced basis approximation \eqref{eq:dualrbm} with snapshots $(t_j)_{j=0}^k$, $W$ an orthonormal basis of $\mathcal{Q}_{k+1}(L^{-1},L^{-1}\mathbf{b})$ with poles $d_j = -t_j^{-1}$, $L_{k+1} = W^TLW$, and $\mathbf{u}_{k+1} = Wf(L_{*,k+1})W^TL^{-1}\mathbf{b}$. Assume $t_j = \infty$ for one $j\in\{0,\dots,k\}$, such that $d_j = -1 / \infty = 0$. Then there holds
	\begin{align*}
		\mathbf{u}_{k+1}^{\textsc{Dual}} = L^{-1}\mathbf{u}_{k+1}.
	\end{align*}
\end{theorem}
\begin{proof}
	We deduce
	\begin{align*}
		\mathcal{V}_{k+1}(L,\mathbf{b}) &= \operatorname{span}\{(t_0I+L)^{-1}\mathbf{b},\dots, (t_kI+L)^{-1}\mathbf{b}\} \\
		&= \operatorname{span}\{(t_0L^{-1}+I)^{-1}L^{-1}\mathbf{b},\dots, (t_kL^{-1}+I)^{-1}L^{-1}\mathbf{b}\} \\
		&= \operatorname{span}\{(L^{-1}+t_0^{-1}I)^{-1}L^{-1}\mathbf{b},\dots, (L^{-1}+t_k^{-1}I)^{-1}L^{-1}\mathbf{b}\},
	\end{align*}
    which affirms that the reduced space $\mathcal{V}_{k+1}(L,\mathbf{b})$ with snapshots $(t_j)_{j=0}^k$ coincides with the rational Krylov space $\mathcal{Q}_{k+1}(L^{-1},L^{-1}\mathbf{b})$ with poles $d_j = -t_j^{-1}$, $j=0,\dots,k$. Let w.l.o.g.\ $t_0 = \infty$, or equivalently, $d_0 = 0$. Then, by definition, $\mathbf{b}\in\mathcal{Q}_{k+1}(L^{-1}, L^{-1}\mathbf{b})$ such that $\mathbf{b} = WW^T\mathbf{b} = VV^T\mathbf{b}$. Since $\mathbf{u}_{k+1}$ is independent of the particular basis, we have
	\begin{align*}
		\textbf{u}_{k+1} = V L_{*,k+1}^{s-2}V^TL^{-1}VV^T\mathbf{b}
		= V L_{*,k+1}^{s-2}L_{*,k+1}V^T\mathbf{b} = V L_{*,k+1}^{s-1}V^T\mathbf{b} = L\mathbf{u}_{k+1}^{\textsc{Dual}}.
	\end{align*}
    This yields $\mathbf{u}_{k+1}^{\textsc{Dual}} = L^{-1}\mathbf{u}_{k+1}$ as claimed.
\end{proof}

\begin{remark}
	From the rational Krylov perspective, a more natural approach to approximate $L^{-s}\mathbf{b}$ would be to directly extract the surrogate from $\mathcal{Q}_{k+1}(L^{-1},L^{-1}\mathbf{b})$ using the poles $d_0 = 0$, $d_j = -t_j^{-1}$ for $j\geq 1$, and $f(z) = z^{s-1}$, or equivalently, $\mathcal{Q}_{k+1}(L^{-1},\mathbf{b})$ with $d_0 = -\infty$, $d_j = -t_j^{-1}$ for $j\geq 1$, and $f(z) = z^{s}$. In this way, the post-processing step, i.e., the final multiplication with $L^{-1}$, could be avoided.
\end{remark}

\subsubsection{Quadrature-based reduced basis methods}

Based on the well-known Dunford-Taylor integral representation
\begin{equation}
    \label{eq:balakrishnan}
	\mathcal{L}^{-s} = \frac{\sin(\pi s)}{\pi}\int_0^\infty t^{-s}(t\operatorname{I} + \mathcal{L})^{-1}\, dt
\end{equation}
for arbitrary positive definite operators $\mathcal{L}$ whose domain is contained in a Hilbert space, Bonito and Pasciak \cite{Bonito2015} presented an exponentially convergent sinc quadrature approximation for $L^{-s}\mathbf{b}$. Using the substitution $y = \ln t$, the method can be summarized as
\begin{align}
	\label{eq:sincappr}
	L^{-s}\mathbf{b} = \frac{\sin(\pi s)}{\pi}\int_{-\infty}^\infty e^{(1-s)y}\mathbf{w}(e^y)\,dy \approx \frac{k_*\sin(\pi s)}{\pi}\sum_{j = -M_s}^{N_s} e^{(1-s)y_j} \mathbf{w}(e^{y_j}),
\end{align}
where $k_*>0$ is a parameter controlling the accuracy of the quadrature, $y_j = jk_*$, and
\begin{align}\label{eq:quadparam}
	M_s = \left\lceil\frac{\pi^2}{(1-s)k_*^2}\right\rceil, \qquad N_s = \left\lceil\frac{\pi^2}{s k_*^2}\right\rceil.
\end{align}
As pointed out in \cite{H:2020}, the method fits in the class of direct rational approximation techniques presented in Section \ref{sec:ratapprox}. In every quadrature node a parametric reaction-diffusion problem of the form \eqref{eq:reacdiff} must be approximated, which turns out to be the method's bottleneck. To alleviate the computational expenses, the authors of \cite{Bonito2020} propose to add an additional layer of approximation in the form of a RBM. Given a collection of snapshots $(t_j)_{j=0}^k$, the surrogate is defined by
\begin{align*}
	\mathbf{u}_{k+1}^{\text{Sinc}} := \frac{k_*\sin(\pi s)}{\pi}\sum_{j = -M_{s_{\max}}}^{N_{s_{\min}}} e^{(1-s)y_j} \mathbf{w}_{k+1}(e^{y_j}),
\end{align*}
where $0<s_{\min} \le s_{\max} < 1$ describes an interval for $s\in[s_{\min},s_{\max}]$ in which we wish to approximate $L^{-s}\mathbf{b}$ efficiently. Due to Theorems \ref{thm:krylov_interpol} and \ref{thm:equivalence}, we have
\begin{align*}
	\mathbf{u}_{k+1}^{\text{Sinc}} = \frac{k_*\sin(\pi s)}{\pi}\sum_{j = -M_{s_{\max}}}^{N_{s_{\min}}} e^{(1-s)y_j}r_j(L)\mathbf{b},
\end{align*}
where $r_j = p_j / q_{k}$, $p_j\in\mathcal{P}_k$, interpolates $f_j(z) := (e^{y_j} + z)^{-1}$ in the eigenvalues of $L_{k+1} = V^TLV$ and $q_k$ is defined by \eqref{eq:q} with $d_j = -t_j.$

The authors of \cite{Antil2019} pursue a similar approach. After algebraic manipulations of \eqref{eq:balakrishnan}, a Gauss-Laguerre quadrature is proposed to discretize the integral, which reads
\begin{align*}
	L^{-s}\mathbf{b} &= \frac{\sin(\pi s_-)}{\pi s_-}\int_0^\infty e^{-y}\mathbf{w}(e^{-\frac{y}{s_-}})\, dy + \frac{\sin(\pi s_+)}{\pi s_+}\int_0^\infty e^{-y}e^{\frac{y}{s_+}}\mathbf{w}(e^{\frac{y}{s_+}}) \, dy \\
	&\approx  \frac{\sin(\pi s_-)}{\pi s_-}\sum_{j=1}^{M_-} \tau_{j,-}\,\mathbf{w}(e^{-\frac{y_{j,-}}{s_-}}) + \frac{\sin(\pi s_+)}{\pi s_+} \sum_{j=1}^{M_+} \tau_{j,+}\, e^{\frac{y_{j,+}}{s_+}}\mathbf{w}(e^{\frac{y_{j,+}}{s_+}}).
\end{align*}
Here, $s_\pm := \frac{1}{2}\pm(s-\frac{1}{2})$ and $(\tau_{j,\pm}, y_{j,\pm})_{j=1}^{M_\pm}$ are the weights and nodes defining the quadrature rule, respectively. The choice
\begin{align*}
	M_+ = \left\lceil\frac{\pi^2}{4sk_*^2}\right\rceil \qquad M_- = \left\lceil\frac{\pi^2}{4(1-s)k_*^2}\right\rceil
\end{align*}
is suggested with parameter $k_*>0$ as in \eqref{eq:quadparam}. A RBM strategy is applied to each of the two sums to reduce the computational costs. Based on two different distributions of snapshots $(t_j^-)_{j = 0}^{k^-}$ and $(t_j^+)_{j = 0}^{k^+}$, $k^\pm\in\mathbb{N}$, together with their respective reduced basis approximations $\mathbf{w}_{k^-\hspace{-0.025cm}+1}^-$ and $\mathbf{w}_{k^+\hspace{-0.025cm}+1}^+$, the surrogate is defined by
\begin{align*}
	\mathbf{u}_{k+1}^{\text{Gauss}} := \frac{\sin(\pi s_-)}{\pi s_-}\sum_{j=1}^{M_-} \tau_{j,-}\,\mathbf{w}_{k^-\hspace{-0.025cm}+1}^-(e^{-\frac{y_{j,-}}{s_-}}) + \frac{\sin(\pi s_+)}{\pi s_+} \sum_{j=1}^{M_+} \tau_{j,+}\, e^{\frac{y_{j,+}}{s_+}}\mathbf{w}_{k^+\hspace{-0.025cm} +1}^+(e^{\frac{y_{j,+}}{s_+}}),
\end{align*}
where $k := k^- + k^+$ and $t_0^- = t_0^+$ is assumed for simplicity. Interpreting the RBM in the above procedure as a corresponding RKM, Theorems \ref{thm:krylov_interpol} and \ref{thm:equivalence} yield the representation
\begin{align*}
	\mathbf{u}_{k+1}^{\text{Gauss}} = \frac{\sin(\pi s_-)}{\pi s_-}\sum_{j=1}^{M_-} \tau_{j,-}\,r_j^-(L)\mathbf{b} + \frac{\sin(\pi s_+)}{\pi s_+} \sum_{j=1}^{M_+} \tau_{j,+}\, e^{\frac{y_{j,+}}{s_+}}r_j^+(L)\mathbf{b},
\end{align*}
where $r_j^\pm = p_j^\pm / q_{k^\pm}^\pm$, $p_j^\pm\in\mathcal{P}_k$, interpolates $f_j^\pm(z) := (e^{\pm \frac{y_{j,\pm}}{s_\pm}} + z)^{-1}$ in the eigenvalues of the corresponding projected operator, and $q_{k^\pm}^\pm$ is defined as in \eqref{eq:q} with $d_j^\pm = -t_j^\pm$, respectively.

We conclude that each of the two quadrature schemes listed above admits a representation as a matrix-vector product of the form $r(L)\mathbf{b}$, where $r$ is a rational function determined by the underlying RBM. Even though the RBM itself allows the interpretation as RKM, $\mathbf{u}_{k+1}^{\textsc{Sinc}}$ and $\mathbf{u}_{k+1}^{\textsc{Gauss}}$ cannot be extracted from a rational Krylov space via Rayleigh-Ritz extraction. The latter can be compensated by applying the RBM directly to the integrand of interest without discretization of the integral itself. To see this, let $\mathcal{V}_{k+1}(L,\mathbf{b})$ refer to an arbitrary reduced space with basis $V$ and $L_{k+1} = V^TLV$. We apply \eqref{eq:balakrishnan} to $L_{k+1}$ to deduce
\[
	L_{k+1}^{-s} = \frac{\sin(\pi s)}{\pi}\int_{0}^\infty t^{-s}(tI_{k+1} + L_{k+1})^{-1} \, dt,
\]
where $I_{k+1}\in\mathbb{R}^{(k+1)\times(k+1)}$ denotes the identity matrix. Hence,
\begin{equation}
	\label{eq:rkm_balakrishnan}
	VL_{k+1}^{-s}V^T\mathbf{b} = \frac{\sin(\pi s)}{\pi}\int_{0}^\infty t^{-s}V(tI_{k+1} + L_{k+1})^{-1}V^T\mathbf{b} \, dt = \frac{\sin(\pi s)}{\pi}\int_{0}^\infty t^{-s}\mathbf{w}_{k+1}(t) \, dt.
\end{equation}
Again, we make use of the transformation $y = \ln(t)$ and invoke \eqref{eq:generic_rbm} to conclude
\begin{align}
	\label{eq:bonito_continuous}
	\mathbf{u}_{k+1}^{\textsc{RB}} =  \frac{\sin(\pi s)}{\pi}\int_{-\infty}^\infty e^{(1-s)y}\mathbf{w}_{k+1}(e^y)\,dy,
\end{align}
if $f(z) = z^{-s}$ in \eqref{eq:generic_rbm}. Similarly, following the idea in \cite[Lemma 3.1]{Antil2019}, one verifies that for this particular choice of $f$
\begin{align}
	\label{eq:antil_continuous}
	\mathbf{u}_{k+1}^{\textsc{RB}} = \frac{\sin(\pi s_-)}{\pi s_-}\int_0^\infty e^{-y}\mathbf{w}_{k+1}(e^{-\frac{y}{s_-}})\, dy + \frac{\sin(\pi s_+)}{\pi s_+}\int_0^\infty e^{-y}e^{\frac{y}{s_+}}\mathbf{w}_{k+1}(e^{\frac{y}{s_+}}) \, dy,
\end{align}
which shows that the quadrature discretization can be omitted when using RBMs. Most notably, this allows us to spare the choice of the particular quadrature as well as the tuning of its associated parameters.

\begin{remark}
	The presented classification of RBMs in fractional diffusion problems is far from complete. E.g., in \cite{Antil2018}, the authors propose to apply a RBM to the extension framework \cite{Caffarelli2007}. The elliptic problem on the artificially extended domain is approximated in a way that makes it amenable to reduced basis technology. It is yet unclear whether this approach allows the interpretation as RKM and requires further investigation.
\end{remark}

It is evident that the performance of all algorithms hinges on a good selection of snapshots (or poles) which determine the underlying matrix $V$. Weak greedy algorithms are among the most popular strategies to provide a good choice for $(t_j)_{j=0}^k$, see, e.g., \cite{DeVore2013}. Provided a computationally efficient error estimator, their aim is to iteratively add those parameters which seemingly yield the largest discrepancy to the exact solution. The authors of \cite{Bonito2020} and \cite{Antil2019} advocate the implementation of such an algorithm combined with a residual-based error estimator to extract the snapshots from the desired parameter domain $\Xi\subset\mathbb{R}$. This approach comes with the benefit of nested spaces, i.e., $\mathcal{V}_k\subset\mathcal{V}_{k+1}$. A difficulty, however, is the fact that the efficient query of $s\mapsto \mathbf{u}_{k+1} \approx L^{-s}\mathbf b$, $\mathbf{u}_{k+1}\in\{\mathbf{u}_{k+1}^{\textsc{Sinc}}, \mathbf{u}_{k+1}^{\textsc{Gauss}}\}$, requires an $s$-independent selection of snapshots and is thus
either limited to proper subsets $s\in[s_{\min},s_{\max}]\subset(0,1)$, or necessitates $\Xi$ to be unbounded. The latter is difficult to tackle numerically. Motivated by our analysis provided in Section \ref{sec:analytical}, these inconveniences might be overcome if one omits the quadrature discretization as in \eqref{eq:rkm_balakrishnan} and chooses $\Xi = \Lambda = [\lambda_1,\lambda_n]$.

A number of algorithms for the adaptive choice of poles in rational Krylov
methods have been proposed as well
\cite{Druskin2010,Druskin2011,Guettel2013b}.
They generally rely on the spectral rational interpolant described in
Theorem~\ref{thm:krylov_interpol} and, unlike the greedy methods described
in the previous paragraph, do not require an error estimator in the
spatial domain, which typically makes their implementation more efficient.
To the best of our knowledge, the performance of these adaptive pole
selection rules for fractional diffusion problems has not been studied.

In contrast, the choice for $(t_j)_{j=0}^k$ proposed in \cite{DS:2019,DS:2020} is given in closed form independently of $\mathbf{b}$ and $s\in(0,1)$. It is based on the so-called Zolotar\"ev points and will be discussed in Section~\ref{sec:analytical} in more detail. Their computation only requires the knowledge of the extremal eigenvalues of $L$. The resulting spaces are not nested, that is, $\mathcal{V}_k\not\subset\mathcal{V}_{k+1}$. However, this drawback can be avoided by constructing the hierarchical sequence of sampling points proposed in \cite{Druskin2009b}, which asymptotically yields the same convergence rates as the ones obtained by Zolotar\"ev.

If the goal is to approximate $L^{-s}\mathbf{b}$ for one fixed value of $s\in(0,1)$, it might be more efficient to choose the low-dimensional space accordingly.
Several RKMs have been proposed which choose poles in dependence of the fractional order $s$; see, e.g., \cite{Aceto2019}.
Theorem \ref{thm:rkm_opt} makes it clear that the question of optimal poles directly relates to the best uniform rational approximation (BURA) of $f(z) = z^{-s}$ in the spectral interval, which has been comprehensively studied throughout the last years in the framework of fractional diffusion \cite{Harizanov2018,Harizanov2018b,Harizanov2020,H:2020}. Until recently, numerical instabilities while computing the BURA were a major obstacle in the availability of optimal poles.
A remedy for this problem was recently proposed in the form of a novel
algorithm for the fast and robust computation of BURAs using only standard
double-precision arithmetic \cite{H:2021}.

\subsection{A rational Krylov method using best-approximation poles}
\label{sec:bura-rkm}

The quasi-optimality result Theorem~\ref{thm:rkm_opt} suggests the use of the
poles of the best uniform rational approximation to $f(z) = z^{-s}$ in the
spectral interval $\Lambda$ as the poles of the rational Krylov method.
Let $r^*_k$ be the rational function of degree at most $k$ which minimizes the
maximum error,
\[
\norm{z^{-s} - r^*_k(z)}_{L_\infty(\Lambda)} = \min_{p,q \in \mathcal P_{k}} \norm{z^{-s} - \frac {p(z)}{q(z)}}_{L_\infty(\Lambda)},
\]
and $(d_j)_{j=0}^k$ its poles. It is a classical result that $r^*_k$ exists and
is unique (see, e.g., \cite{Achieser1992}). We can obtain results on its
approximation quality from the work of Stahl~\cite{Stahl2003}, who has
shown that the best rational approximation $\tilde r^*_k$ to $z^s$ in $[0,1]$
satisfies the error estimate
\[
\norm{z^s - \tilde r^*_k(z)}_{L_\infty[0,1]} \le C_s \exp(-2\pi \sqrt{k s})
\]
with a constant $C_s>0$ which depends on $s$.
Let $\lambda_1>0$ and $\lambda_n$ be the smallest and largest eigenvalues of
$L$, respectively, and define, as in \cite{Harizanov2020},
$r_k(z) := \lambda_1^{-s} \tilde r^*_k(\lambda_1 z^{-1})$.
Then
\[
\norm{z^{-s} - r_k(z)}_{L_\infty(\lambda_1,\lambda_n)}
= \lambda_1^{-s} \left\| \left(\frac {\lambda_1}{z} \right)^{s}
- \tilde r^*_k\left(\frac {\lambda_1}{z} \right) \right\|_{L_\infty(\lambda_1,\lambda_n)}
\le C_s \lambda_1^{-s} \exp(-2\pi \sqrt{k s}).
\]
One easily sees that $r_k$ satisfies the requisite equioscillation conditions
and thus is the best rational approximation to $z^{-s}$ in $[\lambda_1,
\infty)$.  By definition, it follows that
\[
\norm{z^{-s} - r^*_k(z)}_{L_\infty(\Lambda)}
\le \norm{z^{-s} - r_k(z)}_{L_\infty(\Lambda)}
\le C_s \lambda_1^{-s} \exp(-2\pi \sqrt{k s}).
\]
Thus, for the rational Krylov method which approximates
$L^{-s} \mathbf b$ using the poles $(d_j)_{j=0}^k$ of the best rational
approximation, Theorem~\ref{thm:rkm_opt} yields the error estimate
\begin{equation}
\label{eq:bura_error}
    \norm{L^{-s} \mathbf b - \mathbf u_{k+1}} \le
    2 C C_s \lambda_1^{-s} \exp(-2\pi \sqrt{k s}) \norm{\mathbf b}.
\end{equation}
Typically, the smallest eigenvalue (which is closely related to the Poincar\'e
constant of $\Omega$) satisfies $\lambda_1 \ge 1$ but is uniformly bounded with
respect to the discretization parameters, and we can thus ignore the dependence
on $\lambda_1$.

The above estimate makes use only of information on $\lambda_1$. If $\lambda_n$
(or a good bound for it) is known as well, we can directly use the best
rational approximation of $z^{-s}$ on $\Lambda=[\lambda_1,\lambda_n]$ whose
error is smaller than that of the best approximation on $[\lambda_1,\infty)$
and base a rational Krylov method on its poles.
To the best of our knowledge, the analytic behavior of the error of the
best rational approximation to $z^{-s}$ on a finite interval is not known.
For the special case $z \mapsto z^{-1/2}$, the rational function which
minimizes the \emph{relative} maximum error in a finite interval is
explicitly known in terms of elliptic functions \cite{Chiu2002}.
It does not seem that this construction generalizes to different
exponents, however.
Error estimates for certain $(k-1,k)$-Pad\'e approximations to $z^{-s}$ are
given in \cite{Aceto2019b}; very roughly speaking, the authors give estimates
of the order $\sim(k/s)^{-4s}$ in the case of unbounded spectrum and
$\sim\exp(-4k / \sqrt[4]{\kappa})$ with the condition number
$\kappa = \lambda_n /  \lambda_1$ for bounded spectrum.
The latter bound becomes poor as $\kappa\to\infty$ and does not have
the root-exponential bound by Stahl cited above as its limiting case.

Since best rational approximations are usually not known explicitly, they have to
be approximated numerically. The most commonly used algorithm for this task is
the rational Remez algorithm, which is based on the equioscillation property of
the best-approximation error, but is highly numerically unstable in its classical
formulation. To mitigate this problem, extended arithmetic precision has often
been employed (see, e.g., \cite{Varga1992}), which however has the drawback of
high computational effort due to the lack of hardware support for extended
precision arithmetic. Alternate approaches were recently proposed in
\cite{Ionita:PhD,Filip2018}, where new formulations of the Remez algorithm
based on the so-called barycentric rational formula were given, significantly
improving the numerical stability. In particular, the \texttt{minimax} routine
in the latest version of the \texttt{Chebfun} software package \cite{Driscoll2014}
is based on \cite{Filip2018}. Unfortunately, this routine still does not work
well for functions of the type we are interested here. For this purpose, the
second author has recently proposed a novel algorithm for best rational approximation
based on barycentric rational interpolation called BRASIL \cite{H:2021}
which can compute the needed best rational approximations rapidly, to very high
degrees, and using only standard double-precision arithmetic.

\section{Analytical results}
\label{sec:analytical}

Each of the algorithms presented in the previous sections is directly related to rational Krylov or, in the broader sense, rational approximation methods. This changed point of view allows us to use standard techniques from these fields to either provide novel convergence results or illuminate available proofs from a different perspective. We start with the following lemma which is instrumental in the analysis of one of the aforementioned reduced basis schemes.
\begin{lemma}
    \label{lem:rbm_error}
    Let $t\in\mathbb{R}_0^+$, $\mathbf{w}_{k+1}(t)$ the reduced basis approximation of $\mathbf{w}(t)$ with snapshots $(t_j)_{j=0}^k\subset \overline{\mathbb{R}}_0^+$, and $C$ as in Theorem \ref{thm:rkm_opt}. Then there holds for every set $\Sigma := [\lambda_l,\lambda_u]\supset\Lambda$, $\lambda_l>0$,
	\begin{align*}
		\norm{\mathbf{w}_{k+1}(t) - \mathbf{w}(t)} \leq \frac{2C}{t+\lambda_l}\norm{\Theta}_{L_\infty(\Sigma)}\norm{\mathbf b},\qquad \Theta(z) := \prod_{\substack{j=0 \\ t_j\neq \infty}}^k\frac{z-t_j}{z+t_j}.
	\end{align*}
\end{lemma}
\begin{proof}
	The proof follows the outline of \cite[Lemma 5.12]{DS:2019}. Due to Lemma \ref{lemma:equal_searchspace} and Theorem \ref{thm:rkm_opt} we have
	\begin{align*}
		\norm{\mathbf{w}_{k+1}(t) - \mathbf{w}(t)} \leq 2C\norm{\mathbf b}\min_{p\in\mathcal{P}_k}\norm{(t+z)^{-1} - p(z) / q_k(z)}_{L_\infty(\Sigma)},
	\end{align*}
	where $q_k$ is defined as in \eqref{eq:q} with $d_j = -t_j$. Assume for now $t_j = \infty$ for  some $j\in\{0,\dots,k\}$; without loss of generality, we choose $j = 0$. The right-hand side can be bounded by
	\begin{align*}
		\min_{p\in\mathcal{P}_k}\norm{(t+z)^{-1} - p(z) / q_k(z)}_{L_\infty(\Sigma)} \leq \norm{(t+z)^{-1} - \bar{p}(z) / q_k(z)}_{L_\infty(\Sigma)},
	\end{align*}
	where $\bar{p}\in\mathcal{P}_{k-1}$ is uniquely defined by
	\begin{align}
		\label{eq:intpoly}
		\bar{p}(t_i) = (t+t_i)^{-1}q_k(t_i), \rlap{\qquad $i = 1,\dots,k.$}
	\end{align}
	Thanks to this interpolation property, we have that $\bar{p} / q_k$ interpolates $(t+z)^{-1}$ in $t_j$, $j= 1,\dots,k$. Moreover, the difference of both functions is a rational function of degree $(k,k+1)$, such that
	\begin{align*}
		(t+z)^{-1} - \frac{\bar{p}(z)}{q_k(z)} = \frac{c(t)\prod_{j=1}^k(z-t_j)}{(t+z)\prod_{j=1}^k(z+t_j)}
	\end{align*}
	for some $t$-dependent constant $c(t)\in\mathbb{R}$. Multiplying both sides with $(t+z)$ and setting $z=-t$ reveals
	\begin{align*}
		c(t) = \prod_{j=1}^k\frac{t-t_j}{t+t_j}
	\end{align*}
	with absolute value smaller than $1$, which is why the claim holds if $t_0 = \infty$. Otherwise, we can choose $\bar{p}\in\mathcal{P}_k$ according to
	\begin{align*}
		\bar{p}(t_i) = (t+t_i)^{-1}q_k(t_i), \rlap{\qquad $i = 0,\dots,k.$}
	\end{align*}
	Similarly to before, one confirms
	\begin{align*}
		|(z+t)^{-1} - \bar{p}(z) / q_k(z)| \leq \left|\frac{\prod_{j=0}^k(z-t_j)}{(t+z)\prod_{j=0}^k(z+t_j)}\right|,
	\end{align*}
	which proves the claim.
\end{proof}
Instead of applying Theorem \ref{thm:rkm_opt} directly to $f(z) = z^{-s}$, the authors of \cite{DS:2019,DS:2020} aim for a selection of poles according to a (uniform in the parameter $t$) rational approximation of the resolvent function $f(z) = (t+z)^{-1}$. They propose to choose $t_0 = \infty$ and
\begin{align}
	\label{eq:zolopoints}
	t_j = \lambda_n\operatorname{dn}\left(\frac{2(k-j)+1}{2k}K(\delta'),\delta'\right), \qquad \delta' = \sqrt{1-\delta^2}, \qquad \delta = \frac{\lambda_1}{\lambda_n},
\end{align}
for $j= 1,\dots,k$, where $\operatorname{dn}$ denotes the Jacobi elliptic function and $K$ the elliptic integral of first kind; see \cite[Section 16 \& 17]{Abramowitz1964}. These snapshots are a scaled version of the so-called Zolotar\"ev points \cite{Zolotarev1877,Gonchar1969,Oseledets2007}, which are known to minimize the maximal deviation of $\Theta(z)$ in Lemma \ref{lem:rbm_error} over the spectral interval of $L$. As a direct consequence, we obtain exponential convergence for the reduced basis approximation \eqref{eq:generic_rbm} when using \eqref{eq:zolopoints} in the case $f(z) = z^{-s}$, where no analytical result has been available yet.
\begin{theorem}
	\label{thm:convergence_int}
	Let $t_0 = \infty$, $(t_j)_{j=1}^k$ the scaled Zolotar\"ev points from \eqref{eq:zolopoints}, $C$ as in Theorem \ref{thm:rkm_opt}, and $f(z) = z^{-s}$ in \eqref{eq:generic_rbm}. Then there holds for all $s\in(0,1)$
	\begin{align*}
		\norm{\mathbf{u} - \mathbf{u}_{k+1}^\textsc{RB}} \le 4C\lambda_1^{-s}e^{-C^*k}\norm{\mathbf{b}},
	\end{align*}
	where
	\begin{align*}
		C^* = \frac{\pi K(\mu_1)}{4K(\mu)},\qquad \mu = \left(\frac{1-\sqrt{\delta}}{1+\sqrt{\delta}}\right)^2, \qquad \mu_1 = \sqrt{1-\mu^2}, \qquad \delta = \frac{\lambda_1}{\lambda_n}.
	\end{align*}
\end{theorem}
\begin{proof}
	Due to \eqref{eq:balakrishnan} and \eqref{eq:rkm_balakrishnan}, we have that
	\begin{align*}
		\mathbf{u} = \frac{\sin(\pi s)}{\pi}\int_0^\infty t^{-s}(tI+L)^{-1}\mathbf{b}\,dt, \qquad \mathbf{u}_{k+1}^\textsc{RB} = \frac{\sin(\pi s)}{\pi}\int_0^\infty t^{-s} V(tI_{k+1} + L_{k+1})^{-1}V^T\mathbf{b}\, dt.
	\end{align*}
	Lemma \ref{lem:rbm_error} yields
	\begin{align*}
		\norm{\mathbf{u} - \mathbf{u}_{k+1}^\textsc{RB}} &\le \frac{\sin(\pi s)}{\pi}\int_0^\infty t^{-s}\norm{(tI+L)^{-1}\mathbf{b} - V(tI_{k+1} + L_{k+1})^{-1}V^T\mathbf{b}}\,dt \\
		&= \frac{\sin(\pi s)}{\pi}\int_0^\infty t^{-s}\norm{\mathbf{w}(t) - \mathbf{w}_{k+1}(t)}\,dt \\
		&\le 2C\norm{\Theta}_{L_\infty(\Lambda)}\norm{\mathbf b}\,\frac{\sin(\pi s)}{\pi}\int_0^\infty t^{-s}(t+\lambda_1)^{-1}\, dt = 2C\lambda_1^{-s}\norm{\Theta}_{L_\infty(\Lambda)}\norm{\mathbf b},
	\end{align*}
	where the last equality follows from the scalar version of \eqref{eq:balakrishnan}.	With $(t_j)_{j=1}^{k}$ as in \eqref{eq:zolopoints}, we make use of the bound
	\begin{align*}
		\norm{\Theta}_{L_\infty(\Lambda)} \leq 2e^{-C^*k},
	\end{align*}
	a well-known property of the Zolotar\"ev points
    \cite{Zolotarev1877,Gonchar1969,Oseledets2007},
    to complete the proof.
\end{proof}

\section{Numerical Results}
\label{sec:numerics}

This section is devoted to a numerical comparison of the algorithms discussed above, incorporating efficiency, similarities, and performance with respect to several values of the parameter $s$. All methods are implemented in the open source finite element library Netgen/NGSolve\footnote{\url{https://ngsolve.org/}} \cite{Netgen,NGSolve}. We consider the fractional diffusion model problem
\begin{align}
	\label{eq:fraclaplace}
	(-\Delta)^s u = 1\;\;\text{on}\; \Omega, \qquad u = 0\;\;\text{on}\; \partial\Omega,
\end{align}
on the unit square $\Omega := (0,1)^2$ for $s\in\{0.2, 0.5, 0.8\}$. To discretize \eqref{eq:fraclaplace}, we use a finite element space $V_h\subset H_0^1(\Omega)$ constructed over a quasi-uniform triangulation of maximal mesh size $h = 0.008$ and polynomial order $p = 1$. The resulting extremal eigenvalues of $L$ satisfy $\lambda_1\approx 19.74$ and $\lambda_n\approx 560718.48$. For the sake of presentation, we consider only one algorithm from the class of quadrature-based RBMs, namely $\mathbf{u}_{k+1}^{\textsc{Sinc}}$, and omit the dual reduced basis approximation. For a detailed investigation of $\mathbf{u}_{k+1}^{\textsc{Gauss}}$ and $\mathbf{u}_{k+1}^{\textsc{Dual}}$ we refer to \cite{Antil2019} and \cite{DS:2020}, respectively. In favour of comparability, we choose $t_0 = -d_0 = \infty$ in all methods under consideration. The remaining parameters are specified as follows.
\begin{itemize}
	\item For the quadrature-based RBM $\mathbf{u}_{k+1}^{\textsc{Sinc}}$, we set $s_{\min} := 0.2$, $s_{\max} := 0.8$, $k_* := 0.15$, $M_{s_{\max}}$ and $N_{s_{\min}}$ according to \eqref{eq:quadparam}, and $(t_j)_{j = 1}^k\subset[e^{-M_{s_{\max}}k_*}, e^{N_{s_{\min}}k_*}]$ according to the residual-based weak greedy algorithm proposed in \cite{Bonito2020}.
	\item For the rational Krylov approximation, we choose $f(z) = z^{-s}$ and investigate, in view of Theorem \ref{thm:equivalence}, four different configurations of poles or rather snapshots.
		\begin{itemize}
			\item For $\mathbf{u}_{k+1}^{\textsc{Zolo}}$, we choose the snapshots $(t_j)_{j=1}^k$ as scaled Zolotar\"ev points \eqref{eq:zolopoints}. This configuration corresponds to one of the interpolation-based RBM proposed in \cite{DS:2020}. The evaluations of the Jacobi elliptic function and the elliptic integral is performed by means of the special function library provided by \texttt{Scipy}\footnote{\url{https://docs.scipy.org/doc/scipy/reference/special.html}}.

			\item For $\mathbf{u}_{k+1}^{\textsc{Greedy}}$, we choose the same snapshots as for $\mathbf{u}_{k+1}^\textsc{Sinc}$.

            \item For $\mathbf{u}_{k+1}^{\textsc{Jac}}$, we choose the poles $(d_j)_{j=1}^k$ according to the distribution proposed by \cite{Aceto2019}, which is based on a Gauss-Jacobi quadrature approximation for $z^{-s}$.

            \item For $\mathbf{u}_{k+1}^{\textsc{Bura}}$, we choose $(d_j)_{j=1}^k$ according to the BURA poles of $f(z) = z^{-s}$ in $[\lambda_1,\lambda_n]$ obtained by the BRASIL algorithm \cite{H:2021}, which is contained in the \texttt{baryrat}\footnote{\url{https://github.com/c-f-h/baryrat}} open-source Python package developed by the second author.
		\end{itemize}
	\item For the direct rational approximation method $\mathbf{u}_{r} =: \mathbf{u}_{k+1}^{\textsc{Direct}}$, presented in Section \ref{sec:ratapprox}, we choose $r\in\mathcal{R}_{k,k}$ as the best uniform rational approximation of $f(z) = z^{-s}$ on $[\lambda_1,\lambda_n]$ obtained by the BRASIL algorithm \cite{H:2021}.
\end{itemize}

The $L_2$-errors between the expensive discrete solution $\mathbf{u}$ in the sense of \eqref{eq:dem} to \eqref{eq:fraclaplace} and its low-dimensional surrogates obtained by the six methods listed above are reported in Figures \ref{fig:convergence02}, \ref{fig:convergence05}, and \ref{fig:convergence08} for the values of $s = 0.2, 0.5, 0.8$, respectively.

\begin{itemize}
    \item In all cases, exponential convergence can be observed. For the RKM with BURA poles, the rate of convergence is significantly better than predicted by \eqref{eq:bura_error}. One reason for this is the fact that the error bound does not exploit the knowledge of the largest eigenvalue of $L$, but is based on a selection of poles on the unbounded interval $[\lambda_1,\infty)$, as discussed in Section~\ref{sec:bura-rkm}.

	\item For $s = 0.2$, $\mathbf{u}_{k+1}^{\textsc{Sinc}}$, $\mathbf{u}_{k+1}^{\textsc{Zolo}}$, $\mathbf{u}_{k+1}^{\textsc{Greedy}}$, and $\mathbf{u}_{k+1}^{\textsc{Jac}}$ satisfy exponential convergence of order $\mathcal{O}(e^{-C^*k})$, with $C^*$ as in Theorem \ref{thm:convergence_int}. In the case of $\mathbf{u}_{k+1}^{\textsc{Zolo}}$ and $\mathbf{u}_{k+1}^{\textsc{Sinc}}$, this is in accordance with Theorem \ref{thm:convergence_int} and \cite[Lemma 3.3]{Bonito2020}, respectively. For $s\in\{0.5, 0.8\}$, better convergence rates can be observed. A possible explanation for this is the particular choice of $\mathbf{b}$. In \cite{DS:2020} it has already been observed experimentally that for some configurations of the right-hand side, $\mathbf{u}_{k+1}^{\textsc{Zolo}}$ converges with the predicted convergence rate irrespectively of the fractional order.

	\item Those two methods which rely on the BURA, that is, $\mathbf{u}_{k+1}^{\textsc{Bura}}$ and $\mathbf{u}_{k+1}^{\textsc{Direct}}$, provide the best approximation among all tested methods irrespectively of the fractional order. The observed rate of convergence is between $\mathcal{O}(e^{-3.9C^*k})$ and  $\mathcal{O}(e^{-3.6C^*k})$. In view of Theorem \ref{thm:rational_error} and \ref{thm:rkm_opt}, it is not surprising that these methods perform qualitatively similar. What stands out, however, is the observation that the quasi-optimal extraction of $\mathbf{u}_{k+1}^{\textsc{Bura}}$ from $\mathcal{Q}_{k+1}(L,\mathbf{b})$ yields slightly better results than $\mathbf{u}_{k+1}^{\textsc{Direct}}$, which is based on the true BURA of $z^{-s}$. This is due to the fact that the former incorporates information about the right-hand side, which allows the RKM to bias the surrogate towards the particular choice of $\mathbf{b}$. The discrepancy between $\mathbf{u}_{k+1}^{\textsc{Bura}}$ and $\mathbf{u}_{k+1}^{\textsc{Direct}}$ becomes more significant if we choose $b$ in \eqref{eq:fracdiff} sufficiently smooth with homogeneous boundary conditions. In this case, the excitations $\mathbf{b}_j := (\mathbf{b},\mathbf{u}_j)$ of $\mathbf{b}$, with $\mathbf{u}_j$ as in \eqref{eq:discr_eig}, decay quickly, such that $\mathbf{u}_{k+1}^{\textsc{Direct}}$, which assumes a uniform distribution of excitations, requires substantially more linear solves to reach a prescribed accuracy compared to its rational Krylov competitor.

	\item The approximations $\mathbf{u}_{k+1}^{\textsc{Sinc}}$ and $\mathbf{u}_{k+1}^{\textsc{Greedy}}$ coincide for all values of $k$ and $s$. The additional quadrature discretization appears to have no impact on the quality of $\mathbf{u}_{k+1}^{\textsc{Sinc}}$ at all. A possible reason for this might be the fact that the sinc quadrature in \eqref{eq:sincappr} is (close to) exact if we replace $\mathbf{w}$ by $\mathbf{w}_{k+1}$. Indeed, we observe numerically that for any $\lambda\in[\lambda_1,\lambda_n]$
	\begin{align*}
		\int_{-\infty}^\infty e^{(1-s)y}r_y(\lambda) \, dy \approx k_*\sum_{j = -M_{s_{\max}}}^{N_{s_{\min}}}e^{(1-s)y_j}r_{y_j}(\lambda)
	\end{align*}
    up to machine precision, where $r_y$ is the rational function from Theorem \ref{thm:krylov_interpol} with poles in the negative snapshots that interpolates the resolvent function $f_y(z) := (e^y+z)^{-1}$ in the rational Ritz values of the underlying RKM. This exactness property of the quadrature can also be observed for $\mathbf{u}_{k+1}^{\textsc{Gauss}}$. That is, if we greedily sample the snapshots $(t_j)_{j=1}^k$ from a sufficiently large interval such that all three approximations $\mathbf{u}_{k+1}^{\textsc{Greedy}}$, $\mathbf{u}_{k+1}^{\textsc{Sinc}}$, and $\mathbf{u}_{k+1}^{\textsc{Gauss}}$ are built upon that same search space, we observe numerically for, e.g., $s = 0.5$, that $\mathbf{u}_{k+1}^{\textsc{Greedy}} = \mathbf{u}_{k+1}^{\textsc{Sinc}} = \mathbf{u}_{k+1}^{\textsc{Gauss}}$.

	\item As discussed above, the methods based on the BURA provide the most accurate approximation across all scenarios and are specifically tailored towards the fractional parameter $s$. If, however, solutions to \eqref{eq:fraclaplace} for several values of $s$ are required, $\mathbf{u}^{\textsc{Sinc}}$, $\mathbf{u}^{\textsc{Zolo}}$, and $\mathbf{u}^{\textsc{Greedy}}$ outperform their competitors in terms of efficiency since they allow direct querying of the solution for arbitrary $s$ after an initial offline computation phase.
\end{itemize}

\begin{figure}[h!]
    \centering
    \includegraphics[width=0.62\textwidth]{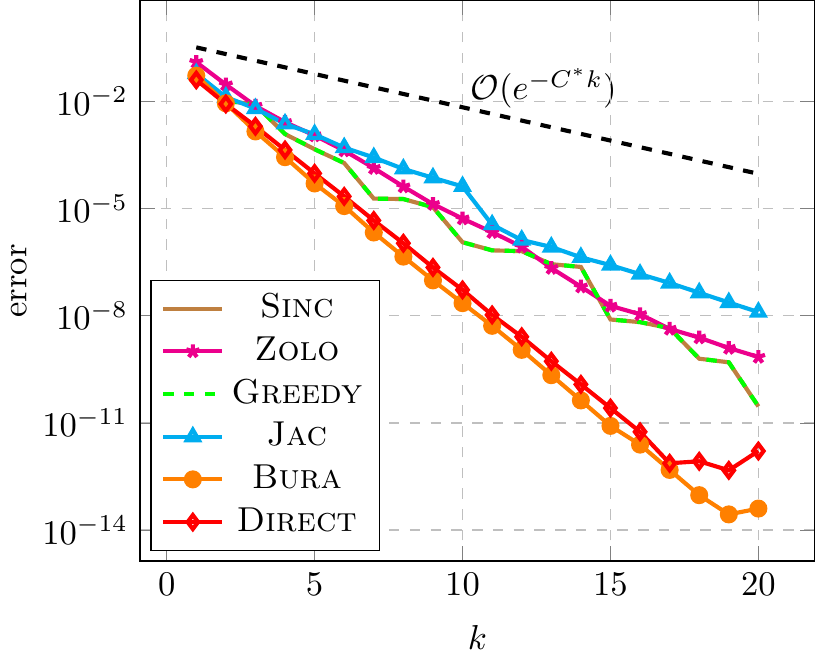}
    \caption{Discrete $L_2$-error $\norm{\mathbf{u} - \mathbf{u}_{k+1}}_{M}$, $M$ as in \eqref{eq:matrices}, for $s = 0.2$, where $\mathbf{u}$ is the discrete high-fidelity solution of \eqref{eq:fraclaplace} and $\mathbf{u}_{k+1}$ the solution obtained by the respective method.}
    \label{fig:convergence02}
\end{figure}
\begin{figure}[h!]
	\centering
	\includegraphics[width=0.62\textwidth]{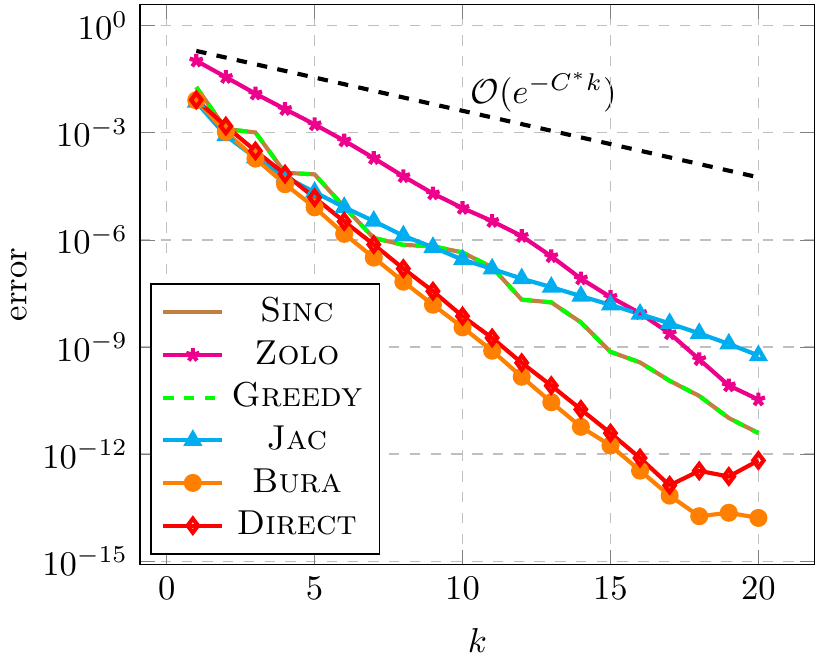}
	\caption{Discrete $L_2$-error $\norm{\mathbf{u} - \mathbf{u}_{k+1}}_{M}$, $M$ as in \eqref{eq:matrices}, for $s = 0.5$, where $\mathbf{u}$ is the discrete high-fidelity solution of \eqref{eq:fraclaplace} and $\mathbf{u}_{k+1}$ the solution obtained by the respective method.}
	\label{fig:convergence05}
\end{figure}
\begin{figure}[h!]
	\centering
	\includegraphics[width=0.62\textwidth]{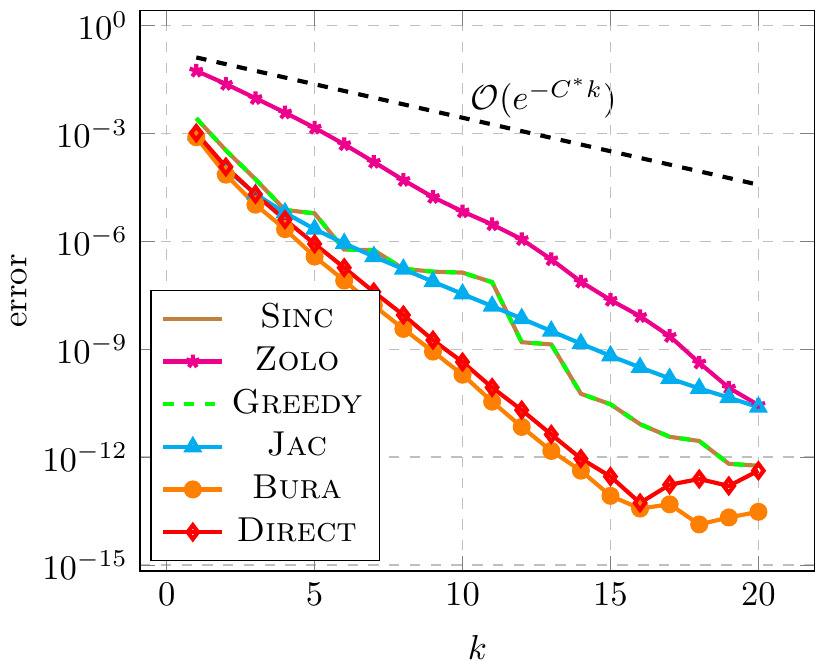}
	\caption{Discrete $L_2$-error $\norm{\mathbf{u} - \mathbf{u}_{k+1}}_{M}$, $M$ as in \eqref{eq:matrices}, for $s = 0.8$, where $\mathbf{u}$ is the discrete high-fidelity solution of \eqref{eq:fraclaplace} and $\mathbf{u}_{k+1}$ the solution obtained by the respective method.}
	\label{fig:convergence08}
\end{figure}

\section*{Acknowledgements}

The first author has been funded by the Austrian Science Fund (FWF) through grant number F 65 and W1245.
The second author has been partially supported by the Austrian Science Fund (FWF)
grant P 33956-NBL.

\bibliographystyle{plainnat}
\bibliography{references}

\end{document}